\documentclass[12pt]{article}

\usepackage[affil-it]{authblk}
\usepackage{amssymb}
\usepackage{amsmath}
\usepackage{amsthm}
\usepackage{amsfonts}
\usepackage{latexsym}
\usepackage{graphicx}
\usepackage{mathrsfs}
\usepackage{tikz}
\usetikzlibrary{arrows}
\usepackage{cite}
\usepackage{float}

\newtheorem{theorem}{Theorem}
\newtheorem{lemma}{Lemma}

\newtheorem{definition}{Definition}
\newtheorem{proposition}{Proposition}

\def\bne{\begin{equation}}
\def\ene{\end{equation}}
\def\bnea{\begin{eqnarray}}
\def\enea{\end{eqnarray}}

\def\bpi{\boldsymbol{\pi}}
\def\bzero{\boldsymbol{0}}

\def\Section#1{Section~\ref{#1}} 
\def\Sections#1#2{Sections~\ref{#1}--\ref{#2}}

\title{Cycle bases of reduced powers of graphs}
\author[a]{\small Richard H.~Hammack}
\author[b]{\small Gregory D.~Smith}

\affil[a]{\small Department of Mathematics and Applied Mathematics,

Virginia Commonwealth University,
Richmond, Virginia, USA, 

rhammack@vcu.edu}    

\affil[b]{\small Department of Applied Science,

The College of William \& Mary,
Williamsburg, VA, USA, 

greg@wm.edu}

\begin{document}
\maketitle 

\begin{abstract}
We define what appears to be a new construction. Given a graph $G$ and
a positive integer $k$, the {\it reduced $k$th power of} $G$, denoted $G^{(k)}$, is the configuration space in which
$k$ indistinguishable tokens are placed on the vertices of $G$, so that any vertex can hold
up to $k$ tokens. 
Two configurations are adjacent if one can be transformed to the other by moving a single
token along an edge to an adjacent vertex. 
We present propositions related to the structural properties of reduced graph powers and, most significantly, provide a construction of  minimum cycle bases of $G^{(k)}$.   

The minimum cycle basis construction is an interesting combinatorial
problem  that is also useful in applications involving configuration spaces. For example, if $G$ is the state-transition graph of a Markov chain model of a stochastic automaton, 
the reduced power $G^{(k)}$ is the state-transition graph for $k$ identical (but not necessarily independent)  automata.   We show how the minimum cycle basis construction  
of $G^{(k)}$ may be used to confirm that state-dependent coupling of automata
does not violate the principle of microscopic reversibility, as required in physical and chemical applications.
\end{abstract}


\noindent Math.~Subj.~Class.: 05C76, 60J27

\section{Introduction}
\label{SECTION:INTRODUCTION}

Time-homogenous Markov chains \cite{Norris97} are used as a mathematical formalism in applications as diverse as computer systems performance analysis \cite{PlateauFourneau91}, queuing theory in operations research \cite{Neuts89}, simulation and analysis of stochastic chemical kinetics \cite{Gillespie07}, and biophysical modeling of  ion channel gating \cite{ColquhounHawkes95}.   

Many properties of a Markov chain, such its rate of mixing and its steady-state probability distribution, can be numerically calculated using its transition matrix \cite{Stewart94}.
A continuous-time Markov chain $X(t)$ ($t\geq 0$) with a finite number of states $\{1,\ldots, \eta \}$ is defined by an initial probability distribution, $\pi_i(0) = \Pr\{ X(0)=i\}$, and a transition matrix  
$Q=(q_{ij})$ where $1 \leq i,j \leq \eta$, $q_{ij} \geq 0$ for $i\neq j$ and $q_{ii}=-\sum_{j\neq i} q_{ij}$, so called because, for $i\neq j$, 
$q_{ij} =  \lim_{dt \rightarrow 0 }\Pr \{ X(t+dt)=j | X(t)= i \} / dt$.  The requirement that $Q$ has zero row sums, $\sum_{j} q_{ij} = 0$, corresponds to 
conservation of probability, $\sum_i \pi_i(t) = 1$, in the ordinary differential equation initial value problem, $d\bpi/dt=\bpi Q$ with initial condition $\bpi(0)$, solved by the time-dependent discrete probability distribution $\bpi(t) = ( \pi_1(t), \ldots, \pi_\eta(t) )$ where $\pi_i(t) = \Pr\{ X(t)=i\}$.

A CTMC with a single communicating class of $\eta < \infty$ states is irreducible, positive recurrent, and has a unique steady-state probability distribution that solves 
$\bar{\bpi} Q = \bzero$ subject to $\sum_i \bar{\pi}_i = 1$ (by the Perron-Frobenius theorem).
The Perron vector and steady-state distribution $\bar{\bpi}$ is the limiting probability distribution of the Markov chain, $\lim_{t\rightarrow \infty} \| \bpi(t) - \bar{\bpi} \| = 0$, for any initial condition satisfying conservation of probability, $\sum_i \pi_i(0) = 1$.  In general, the calculation of steady-state distributions and other properties for  Markov chains with $\eta$ states requires algorithms of $\mathcal{O}(\eta^3)$ complexity.
 
Many open questions in the physical and biological sciences involve the analysis of systems that are naturally modeled as a collection of interacting stochastic automata \cite{Liggett99,BallEtal00,RichouxVerghese11}.  Unfortunately, representing a 
stochastic automata network as a single {\it master} Markov chain 
suffers from the computational limitation that the aggregate number of states is exponential in the number of components.  For example, the transition matrix for $k$ coupled stochastic automata, each of which can be represented by an $v$-state Markov chain, has $\eta = v^k$ states and requires algorithms of  $\mathcal{O}(v^{3k})$ complexity.

Many results are relevant to overcoming combinatorial state-space explosions of coupled stochastic automata.  For example,  memory-efficient numerical methods may use ordinary Kronecker representations of the master transition matrix $Q = \sum_\ell \bigotimes_{n=1}^k R_{\ell n}$ where the $R_{\ell n}$ are size $v$, and many are identity matrices, eliminating the need to generate and store the size $v^k$  transition matrix \cite{CamposEtal99}.
Kronecker representations may be generalized to allow for matrix operands whose entries are functions that describe state-dependent transition rates, i.e., $Q = \bigoplus_{n=1}^k F_n$ and $F _n(i,j) : \times_{n=1}^{k} X_n \rightarrow \mathbb{R}$ where $X_n(t)$ is the state of the $n$th automata \cite{BenoitEtal04a}.  
Hierarchical Markovian models may be derived in an automated manner  and leveraged by multi-level numerical methods \cite{Buchholz95}.

Redundancy in master Markov chains for interacting stochastic automata   can often be eliminated without approximation.  Both lumpability at the level of individual automata and model composition  have been extensively researched, though the latter 
 reduces  the state space in a manner that eliminates Kronecker structure  \cite{Buchholz94,BenoitEtal04b,GusakEtal03a}.  To see this, consider $k$ identical and indistinguishable stochastic automata, each with $v$ states, that interact via transition rates that are functions of the global state, that is, $Q = \bigoplus^k F$ where $F(i,j) : \times_{\ell=1}^{v} \,n_\ell \rightarrow \mathbb{R}$ where $n_\ell(t)$ = $\sum_{n=1}^k \mathcal{I}\{ X_n(t)=\ell\}$ is the number of automata in state $\ell$.  As defined $Q$, is size $v^k$, however,  states may be lumped using  symmetry in the model specification to yield an equivalent master Markov chain of size $\eta = \binom{k+v-1}{k}$.  Although model reduction in this spirit is intuitive and widely used in applications, the mathematical structure of the transition graphs resulting from such contractions does not appear to have been extensively studied.  

More concretely, let $G$ represent the transition graph for an $v$-state automaton with transition matrix $Q=(q_{ij})$.   As required in many applications, we assume that $Q$ is irreducible and that state transitions are reversible ($q_{ij}> 0 \Leftrightarrow q_{ji}>0$, $i \neq j$).  Thus, the transition graph $G$ corresponding to $Q$ is simple (unweighted, undirected, no loops or multiple edges) and connected (by the irreducibility of $Q$).  The transition graph $G$ has adjacency matrix $A(G)=(a_{ij})$ where $a_{ii}=0$, and for $i \neq j$, $a_{ij}=0$ when $q_{ij}=0$ and $a_{ij}=1$ when $q_{ij}>0$.   

The transition graph for the master Markov chain for $k$ automata with transition graphs $G_n$ is the Cartesian graph product $G_1\Box G_2\Box\cdots\Box G_k$.  If these $k$ automata are identical, 
the transition graph for the master Markov chain is the $k$th Cartesian power of $G$, that is, the $k$-fold product $G^k=G\Box G\Box\cdots\Box G$.   The focus of this paper is the \textbf{$\boldsymbol{k}$-th reduced  power} of~$G$, i.e., the transition graph of the contracted master Markov chain for $k$ indistinguishable (but not necessarily independent)  $v$-state automata with isomorphic transition graphs.

The remainder of this paper is organized as follows.
In \Sections{SECTION:REDUCED}{SECTION:CONFIGURATION} 
we formally define the reduced power of a graph and interpret it as particular configuration space.    \Sections{SECTION:MCB}{SECTION:HIGHER}  present our primary result, the construction of minimal cycle bases of reduced graph powers.  \Section{Sec:Discussion} explicates the relevance of these minimal cycle bases to applications that do not allow state-dependent coupling of  automata  to introduce nonequilibrium steady states. 

\section{Reduced Cartesian powers of a graph}
\label{SECTION:REDUCED}

There are several equivalent formulations of the reduced  power of a graph. For the first formulation, recall that 
given graphs $G$ and $H$,  their {\it Cartesian product}  is the graph $G\Box H$ whose
vertex set is the Cartesian product $V(G)\times V(H)$ of
the vertex sets of $G$ and $H$, and whose  edge set is
$$\begin{array}{lcl}
E(G\Box H) & = & \big\{ (x,u)(y,v) \mid xy\in E(G) \mbox{ and } u=v,\; \mbox{\bf or } \;x=y \mbox{ and }  uv\in E(H)\big\}.
\end{array}$$
This product is commutative and associative \cite{HammackEtal2011}.
For typographical efficiency we may abbreviate a vertex $(x,y)$ of
$G\Box H$ as $xy$ if there is no danger of confusion.

The $k$th {\em Cartesian power} of a graph $G$ is the $k$-fold product $G^k=G\Box G\Box\cdots\Box G$.
The symmetric group $S_k$ acts on $G^k$ by permuting the factors. Specifically, for a
permutation $\pi\in S_k$ the map
$$(x_1,x_2,\ldots, x_k)\mapsto (x_{\pi(1)},x_{\pi(2)},\ldots, x_{\pi(k)})$$
is an automorphism of $G^k$. The $k$th {\em reduced power} of $G$ is the graph that has as vertices
the orbits of this action, with two orbits being adjacent if $G^k$ has an edge joining one
orbit to the other. Said more succinctly, the reduced $k$th power is the quotient
$G^k/S_k$ of $G^k$ by its $S_k$ action.

Figure~\ref{Fig:Grid} shows a graph $G$ next to $G^2=G\Box G$. 
The $S_2$ action on $G^2$  has as orbits the singletons $\{aa\}$, $\{bb\}$, $\{cc\}$, $\{dd\}$,
along with the pairs $\{ab,ba\}$, $\{ac,ca\}$, $\{ad,da\}$, $\{bc,cb\}$, $\{bd,db\}$, and $\{cd,dc\}$.
Let us identify a singleton orbit such as $\{aa\}$ with the monomial $aa=a^2$, and a paired orbit
such as $\{ab,ba\}$ with the monomial $ab$ (with $ab=ba$). The reduced power $G^{(2)}$ appears on the right
of Figure~\ref{Fig:Grid}. Note that two monomials $xy$ and $uv$ are adjacent in $G^{(2)}$
provided that $xy$ and $uv$ have a common factor, and the remaining two factors are adjacent
vertices of $G$.

\begin{figure}[t!]
\centering
\includegraphics{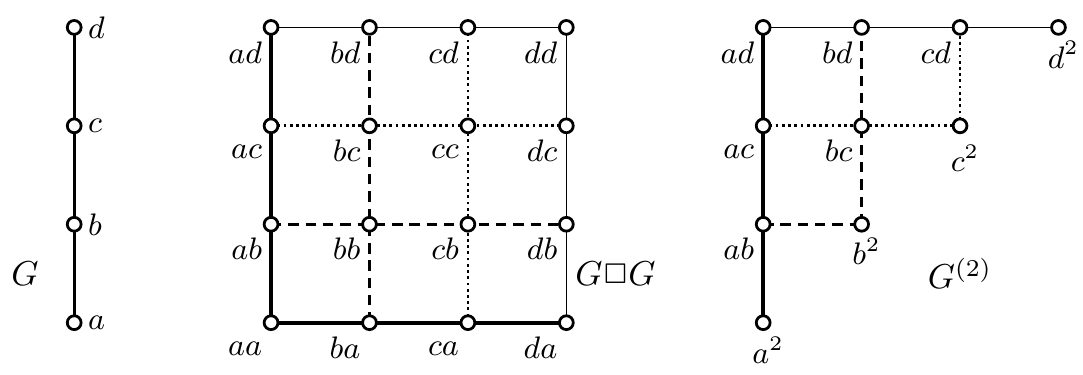}
\caption{{\small A graph $G$, the Cartesian square $G^2=G\Box G$,
and the reduced power $G^{(2)}$.
For each $x\in V(G)$, the vertices $\{xv\mid v\in V(G)\}$ induce a subgraph $Gx\cong G$ of $G^{(2)}$.
These subgraphs are shown dashed, dotted and solid in $G^{(2)}$.
Note $Gx$ and $Gy$ intersect precisely at  vertex $xy$ if $x\ne y$.}}
\label{Fig:Grid}
\end{figure}

As each monomial $xy$ corresponds uniquely to the $2$-multiset $\{x,y\}$ of vertices of~$G$,
we can also define the reduced power $G^{(2)}$ as follows. Its vertices are the
$2$-multisets of vertices of $G$, with two multisets being adjacent precisely if they agree in
one element, and the other elements are adjacent in $G$.

In general, higher reduced powers $G^{(k)}$ can be understood as follows.
Suppose $V(G)=\{a_1,a_2,\ldots,a_v\}$. Any vertex of $G^{(k)}$ is the
$S_k$-orbit of some $x=(x_1,x_2,\ldots,x_k)\in V(G^k)$.
For each index $1\leq i\leq v$, say $x$ has $n_i\geq 0$ coordinates equal to $a_i$.
Then $\sum_{i=1}^v n_i=k$, and the $S_k$-orbit of $x$ consists precisely of 
those $k$-tuples in $V(G^k)$ having $n_i$ coordinates equal to $a_i$, for $1\leq i\leq v$.
This orbit --- this vertex of $G^{(k)}$ --- can then be identified with either the degree-$k$ monomial
$$a_1^{n_1}a_2^{n_2}\cdots a_v^{n_v},$$
 or with the $k$-multiset
 \begin{equation}
 \{\,\underbrace{a_1,a_1,\ldots, a_1}_{n_1} \mid
 \underbrace{a_2,a_2,\ldots, a_2}_{n_2} \mid\ldots\; \ldots \mid
  \underbrace{a_v,a_v,\ldots, a_v}_{n_v}\,
 \},
 \label{MULTI}
 \end{equation}
where $v-1$ dividing bars are inserted for  clarity.
We will mostly use the monomial notation for $V(G^{(k)})$, but will
also employ the multiset phrasing when convenient.
Let us denote the set of monic monomials of degree~$k$, with indeterminates $V(G)$,
as $M_k(G)$, with $M_0(G)=\{1\}$.
The above, together with the
definition of the Cartesian product, yields the following.

\begin{definition}
For a graph $G$ with vertex set $\{a_1,a_2,\ldots,a_v\}$, the {\bf reduced  $\boldsymbol{k}$th power}
$G^{(k)}$ is the graph whose vertices are the monomials
$a_1^{n_1}a_2^{n_2}\cdots a_v^{n_v}\in M_k(G)$. For edges,
if $a_ia_j$ is an edge of $G$, and $f(a_1,a_2,\ldots a_v)\in M_{k-1}(G)$, then\
$a_i f(a_1,a_2,\ldots,a_v)$ is adjacent to $a_j f(a_1,a_2,\ldots, a_v)$.
\label{DEF:REDUCED}
\end{definition}

Figure~\ref{Fig:Three} shows the three-cycle $G=C_3$ and its
reduced  second and third powers. Figure~\ref{Fig:FiveCycle} shows the five-cycle and its
reduced second and third powers.  

The reduced power $G^{(k)}$ is not to be confused with the {\it symmetric power} of $G$,
for which each vertex represents a $k$-subset of $V(G)$, and two $k$-subsets are joined
if and only if their symmetric difference is an edge of $G$ \cite{AlzagaEtal10,AudenaertEtal07}.  

\begin{figure}[t!]
\centering
\includegraphics{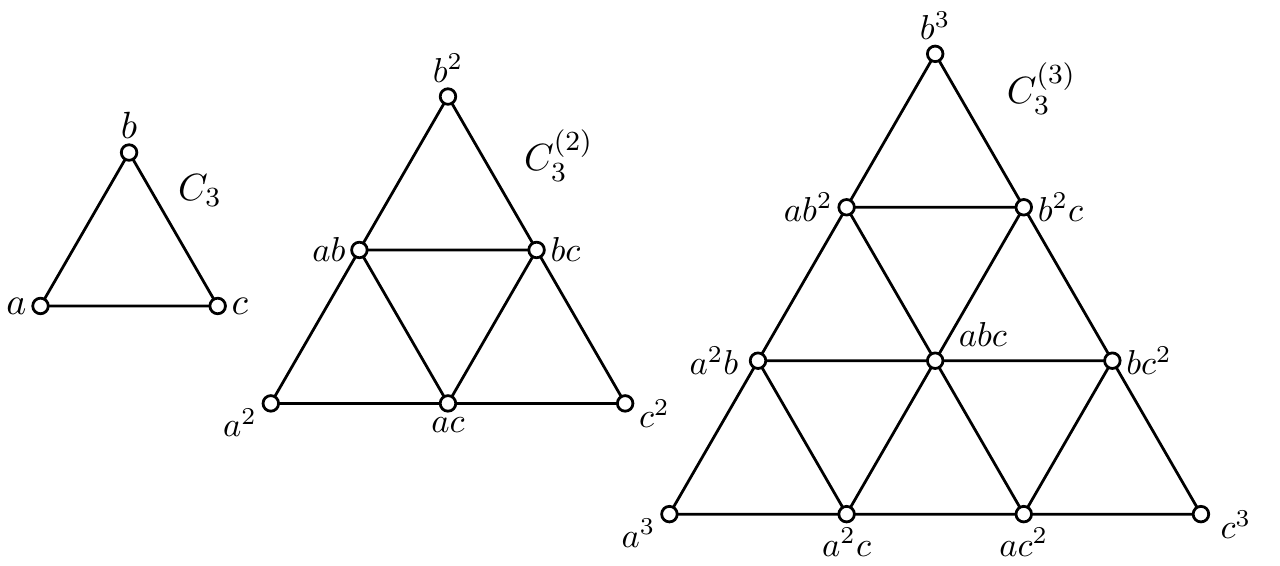}
\caption{The three-cycle $C_3$ and its second and third reduced powers $C_3^{(2)}$ and $C_3^{(3)}$.}
\label{Fig:Three}
\end{figure}

The multiset notation~(\ref{MULTI}) gives a quick formula for the number of vertices of 
reduced $k$th powers.
This presentation describes the multiset as a list of length $k+v-1$ involving $k$ symbols $a_i$, $1 \leq i \leq k$,
and $v-1$ separating bars. 
We can count the multisets by choosing $k$ slots for the $a_i$'s and filling in the remaining slots
with bars. Therefore, when $|V(G)|=v$,
\begin{equation}
\left| V\big(G^{(k)} \big)\right| = \binom{k+v-1}{k}.
\label{EQN:VERTICES}
\end{equation}
The number of vertices in $G^k$ that are identified with vertex $a_1^{n_1}a_2^{n_2}\cdots a_v^{n_v} \in V(G^{(k)})$ in the quotient $G^{(k)}=G^k/S_k$ is given by the multinomial coefficient $\binom{k}{n_1\; n_2\;\ldots\;n_v}$ .

Definition~\ref{DEF:REDUCED} says that for each edge $a_ia_j$ of $G$,
and for each monomial $f\in M_{k-1}(G)$, there is an
edge of $G^{(k)}$  from $a_i f$ to $a_j f$.
Because there are $\binom{k+m-2}{k-1}$ such monomials $f$,
\begin{equation}
\left| E\big(G^{(k)} \big)\right| =|E(G)|\cdot \binom{k+v-2}{k-1}.
\label{EQN:EDGES}
\end{equation}

\section{Reduced graph powers as configuration spaces}
\label{SECTION:CONFIGURATION}

The reduced power $G^{(k)}$ 
is the transition graph of the contracted master Markov chain for  $k$ identical and indistinguishable $v$-state automata, each with transition graph $G$.
Consequently, an intuitive way of envisioning $G^{(k)}$ is to imagine it as a configuration space in which
$k$ indistinguishable tokens are placed on the vertices of $G$, so that any vertex can hold
up to $k$ tokens. The monomial $a_1^{n_1}a_2^{n_2}\cdots a_v^{n_v}$ then
represents the configuration in which $n_i$ tokens are placed on each vertex $a_i$.
Two configurations are adjacent if one can be transformed to the other by moving a single
token along an edge of $G$ to an adjacent vertex.  In this way  $G^{(k)}$ is interpreted as the space of all such configurations.   See~\cite{Fabila2012} for a related construction in which no vertex can hold more than one token. 

The reduced power $G^{(k)}$ may also be interpreted the {\em reachability graph} for a fundamental class of stochastic Petri nets 
with $k$ tokens, $v=|V(G)|$ places, and $2|E(G)|$ flow relations (directed arcs) between places \cite{BuchholzKemper02,Reisig13}.   The arc from place $a_i$ (origin) to place $a_j$ (destination) has firing rate $n_i q_{ij}$ given by the product of transition rate $q_{ij}$ and the number  $n_i$ of tokens  in the origin place.  That is, the $a_i \rightarrow a_j$ firing time is the minimum of $n_i$ exponentially distributed random variables with expectation $1/q_{ij}$.  The $a_i \rightarrow a_j$ firing rate per token will be denoted $q_{ij}[a_1^{n_1}a_2^{n_2}\cdots a_v^{n_v}] $ when it is a function of the global state (token configuration) of the stochastic Petri net.

The token interpretation can be helpful in deducing properties of reduced powers, such as
the following. 
\begin{proposition}
\label{PROP:VERTEXDEGREE}
\label{EQN:DEGREE}
The vertex $a_1^{n_1}a_2^{n_2}\cdots a_v^{n_v}$ of $G^{(k)}$ has degree
\begin{equation*}
\deg\big(a_1^{n_1}a_2^{n_2}\cdots a_v^{n_v}\big)=\sum_{n_i\geq 1}\deg_G(a_i).
\end{equation*}
\end{proposition}
\begin{proof}
The configuration $a_1^{n_1}a_2^{n_2}\cdots a_v^{n_v}$ can be transformed to an adjacent
configuration only by moving a token on some vertex $a_i$ (with $n_i\geq 1$) to an adjacent vertex.
\end{proof}
\begin{figure}[h!]
\centering
\includegraphics{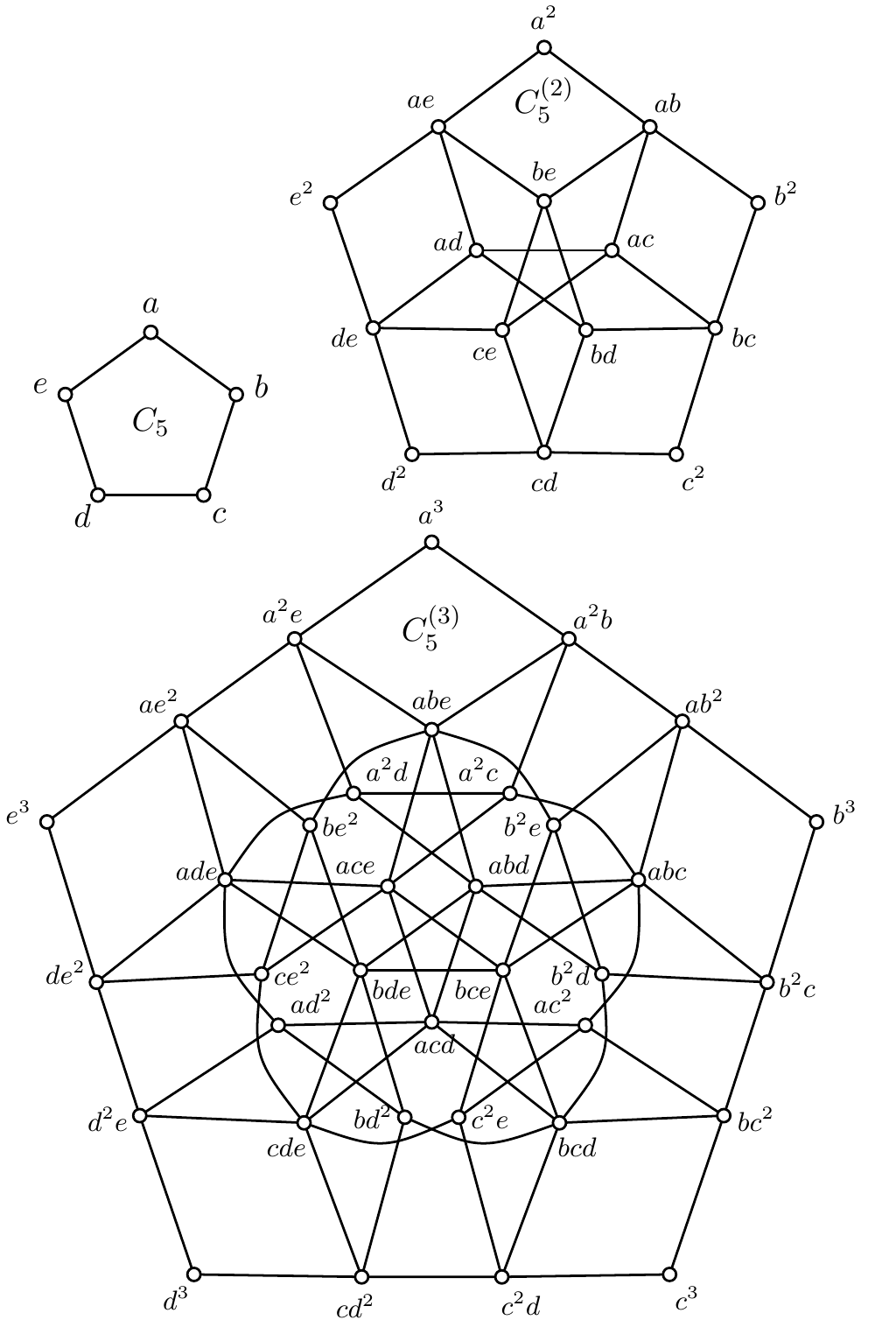}
\caption{The five-cycle $C_5$ and its second and third reduced powers
$C_5^{(2)}$ and $C_5^{(3)}$.}
\label{Fig:FiveCycle}
\end{figure}

\section{Cycle bases and minimum cycle bases}
\label{SECTION:MCB}

\begin{figure}[b!]
\centering
\includegraphics{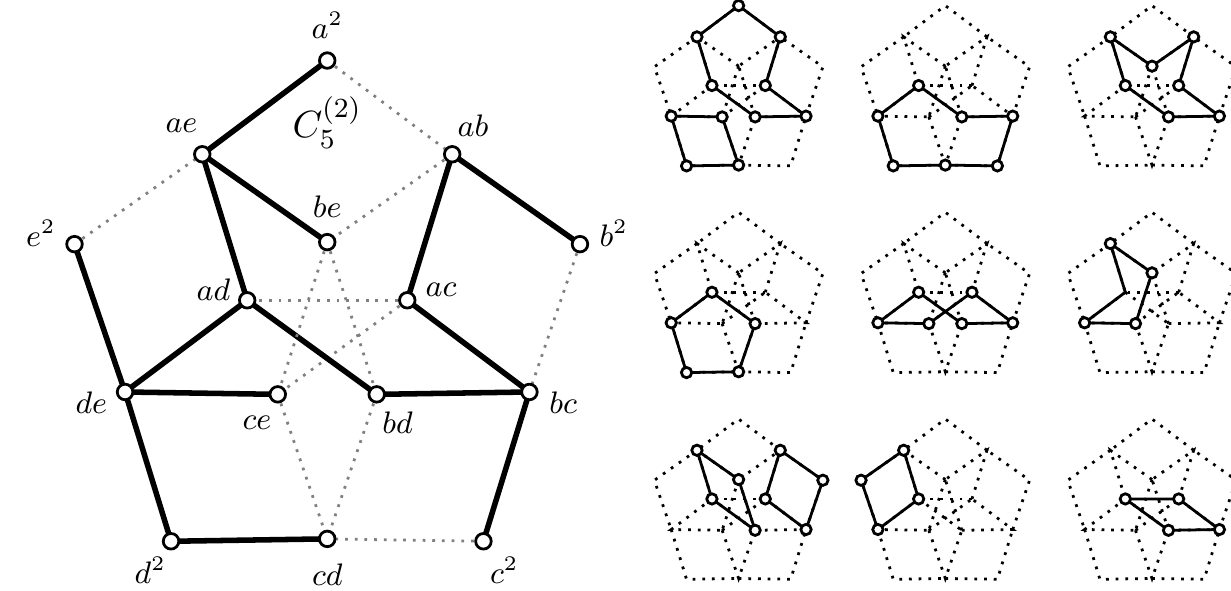}
\caption{A spanning tree $T$ of $G=C_5^{(2)}$.  The set $S=E(G)-E(T)$ has $\beta(G)= 25-15+1=11$ edges. 
For each $e\in S$, let $C_e$ be
the unique cycle in $T+e$. The set $\{C_e\ | \ e\in S\}$  is a cycle basis for $G$, but not a minimum cycle basis (see~Figure~\ref{Fig:FiveCycleMCB}).}
\label{Fig:FiveCycleTree}
\end{figure}

Here we quickly review the fundamentals of cycle spaces and bases. The following is
condensed from Chapter 29 of~\cite{HammackEtal2011}.

For a graph $G$, its {\it edge space} $\mathscr{E}(G)$ is the power set of $E(G)$ viewed
as a vector space over the two-element field $\mathbb{F}_2=\{0,1\}$, where the zero vector is 
$0=\emptyset$ and addition is symmetric difference.
Any vector $X\in\mathscr{E}(G)$ is viewed as the subgraph of $G$ induced on $X$, so
$\mathscr{E}(G)$ is the set of all subgraphs of $G$ without isolated vertices.
Thus $E(G)$ is a basis for $\mathscr{E}(G)$, and $\dim(\mathscr{E}(G))=|E(G)|$.
The {\it vertex space} $\mathscr{V}(G)$ of $G$ is the power set of $V(G)$ as
a vector space over $\mathbb{F}_2$. It is the set of all edgeless subgraphs
of $G$ and its dimension is $|V(G)|$.

We define a linear {\it boundary map} $\delta_G:\mathscr{E}(G)\to \mathscr{V}(G)$ by declaring
that $\delta_G(xy)=x+y$ on the basis $E(G)$.
The subspace $\mathscr{C}(G)=\ker(\delta_G)$ is called the
{\it cycle space} of $G$. It contains precisely the subgraphs
in $\mathscr{E}(G)$ whose vertices all have even degree (that is, the Eulerian
subgraphs).  Because every such subgraph
can be decomposed into edge-disjoint cycles, each in $\mathscr{C}(G)$,
we see that  $\mathscr{C}(G)\subseteq$
$\mathscr{E}(G)$ is spanned by the cycles in $G$. 

The dimension of $\mathscr{C}(G)$, denoted $\beta(G)$, is called the (first)
{\it Betti number} of $G$.
If $G$ is connected, the rank theorem applied to $\delta_G$ yields
\begin{equation}
\label{Eqn:Cycles:Beta}
\beta(G)=|E(G)|-|V(G)|+1.
\end{equation}

A basis for the cycle space is called a {\it cycle basis}.
To make a cycle basis of a connected graph $G$, take a spanning tree
$T$, so the set $S=E(G)-E(T)$ has $|E(G)|-|V(G)|+1=\beta(G)$ edges. 
For each $e\in S$, let $C_e$ be
the unique cycle in $T+e$. Then the set $\mathscr{B}=$ $\{C_e\ | \ e\in S\}$ is linearly
independent.
As $\mathscr{B}$ has cardinality $\beta(G)$, it is a basis (see Figure~\ref{Fig:FiveCycleTree}).

The elements of a cycle basis are naturally weighted by their number of edges.
The {\it total length} of a cycle basis $\mathscr{B}$ is the number
$\ell(\mathscr{B})=\sum_{C\in\mathscr{B}}|C|$. A cycle basis with
the smallest possible total length is called a {\it minimum cycle basis}, or {\it MCB}.

The cycle space is a weighted matroid where each element $C$ has weight $|C|$.  Hence the
Greedy Algorithm \cite{Oxley1992} always terminates with an MCB: Begin with 
$\mathscr{M}=\emptyset$; then append shortest cycles to it, maintaining independence of
$\mathscr{M}$, until no further shortest  cycles can be appended; then append next-shortest
cycles, maintaining independence, until no further such cycles can be appended; and so on, until
$\mathscr{M}$ is a maximal independent set. Then $\mathscr{M}$ is an MCB.

Here is our primary criterion for determining if a cycle basis is an MCB. (See Exercise 29.4
of~\cite{HammackEtal2011}.)

\begin{proposition}
\label{PROP:TEST}
A cycle basis $\mathscr{B}$ $=$ $\{B_1, B_2, \ldots, B_{\beta(G)}\}$ for a graph $G$
is an MCB if and only if every $C$  $\in$ $\mathscr{C}(G)$ is a sum of basis elements whose
lengths do not exceed $|C|$. 
\end{proposition}

For graphs $G$ and $H$, a {\it weak homomorphism} $\varphi:G\to H$ is
a map $\varphi:V(G)\to V(H)$ having the property that for each $xy$ of $G$, either
$\varphi(x)\varphi(y)$ is an edge of $H$, or $\varphi(x)=\varphi(y)$. Such a map induces a linear
map $\varphi^*:\mathscr{E}(G)\to \mathscr{E}(H)$ defined on the basis $E(G)$ as
$\varphi^*(xy)=\varphi(x)\varphi(y)$ provided $\varphi(x)\ne\varphi(y)$, and 
$\varphi^*(xy)=0$ otherwise. Similarly we define 
$\varphi_V^*:\mathscr{V}(G)\to \mathscr{V}(H)$ as $\varphi_V^*(x)=\varphi(x)$ on the basis $V(G)$.
Thus we have the following commutative diagram. (Check it on the basis $E(G)$.)
\begin{center}
\begin{tikzpicture}[scale=1,style=thick]
\def\len{3}
\def\hgt{1.25}
\draw (0,0) node  (UL) {$\mathscr{E}(G)$};
\draw (\len,0) node (UR) {$\mathscr{E}(H)$};
\draw (0,-\hgt) node  (LL) {$\mathscr{V}(G)$};
\draw (\len,-\hgt) node (LR) {$\mathscr{V}(H)$};
\draw [->] (UL) to node [left] {$\delta_G$} (LL);
\draw [->] (UR) to node [right] {$\delta_H$} (LR);
\draw [->] (UL) to node [above] {$\varphi^*$} (UR);
\draw [->] (LL) to node [above] {$\varphi_V^*$} (LR);
\end{tikzpicture}
\end{center}

From this, $\varphi^*$ restricts to a map $\mathscr{C}(G)\to \mathscr{C}(H)$
on cycle spaces, because if $C\in \mathscr{C}(G)$, then $\delta_G(C)=0$, whence
$\delta_H\varphi^*(C)=\varphi_V^*\delta_G(C)=0$, meaning $\varphi^*(C)\in\ker(\delta_H)=$
$\mathscr{C}(H)$. Certainly if $\varphi$ is a graph isomorphism, then $\varphi^*$ is a
vector space isomorphism. 

Of special interest will be the projections $p_i:G^k\to G$, where
$p_i(x_1,x_2,\ldots,x_k)=x_i$. These are weak homomorphisms and hence induce
linear maps  $p_i^*:\mathscr{C}(G^k)\to \mathscr{C}(G)$.

Another important map is the natural projection $\eta:G^k\to G^{(k)}$ sending each $k$-tuple
$x=(x_1,x_2,\ldots,x_k)$ to the monomial representing the $S_k$-orbit containing $x$. This 
map $\eta^*$ also is a weak homomorphism, inducing a linear map $\eta^*:\mathscr{C}(G^k)\to \mathscr{C}(G^{(k)})$.

\begin{lemma}
If $G$ is connected, the map $\eta^*:\mathscr{C}(G^k)\to \mathscr{C}(G^{(k)})$ is surjective. 
\label{LEMMA:ETA}
\end{lemma}
\begin{proof} 
Because any element of $\mathscr{C}(G^{(k)})$ is an edge-disjoint union of cycles,
it suffices to show that any
cycle $C=f_0f_1\cdots f_n f_0\in \mathscr{C}(G^{(k)})$ equals $\eta^*(C')$ for some
$C'\in \mathscr{C}(G^{k})$. For each index $i$, let $x_iy_{i+1}\in E(G^k)$ be an edge for which
$\eta^*(x_iy_{i+1})=\eta(x_i)\eta(y_{i+1})=f_if_{i+1}$. (Each $x_i$, $y_i$ is a $k$-tuple, and
index arithmetic is modulo $n$.)
Note that $\eta(x_i)=\eta(y_i)$, meaning $x_i$ and $y_i$ are in the same $S_k$-orbit, that is,
$y_i$ equals $x_i$ with its coordinates permuted. 

We will argue that each pair $y_i$, $x_i$ can be joined by a path
$P_i$ in $G^k$, with $\eta^*(P_i)=0$. This will prove the lemma because then 
$$C'=P_0+x_0y_1+P_1+x_1y_2+P_2+\ldots +P_n+x_ny_0\;\in \;\mathscr{C}(G^k)$$
satisfies $\eta^*(C')=C$.

Consider two vertices $(\ldots a\ldots b\ldots)$ and $(\ldots b\ldots a\ldots)$ of $G^k$
that are identical except for the transposition of two coordinates $a$ and $b$. Take a path
$a=v_0v_1\cdots v_q=b$ from $a$ to $b$ in $G$. Now form the following two paths in $G^k$
$$\begin{array}{ccl}
Q &=& (\ldots a\ldots b\ldots)(\ldots v_1\ldots b\ldots)(\ldots v_2\ldots b\ldots)\ldots (\ldots b\ldots b\ldots)\\
R &=& (\ldots b\ldots a\ldots)(\ldots b\ldots v_1\ldots)(\ldots b\ldots v_2\ldots)\ldots (\ldots b\ldots b\ldots).
\end{array}$$
Concatenation of $Q$ with the reverse of $R$ is a path from 
$(\ldots a\ldots b\ldots)$ to $(\ldots b\ldots a\ldots)$. Moreover $\eta^*(Q+R)=0$ because
the images of the $j$th edges of $Q$ and $R$ are always equal; hence the edges cancel in
pairs. As $y_i$ and $x_i$ differ only by a sequence of transpositions of their coordinates,
the above construction can be used to build up a path $P_i$ from $y_i$ to $x_i$ with
$\eta(P_i)=0$.
\end{proof}

\begin{figure}[b!]
\centering
\includegraphics[width=0.98\textwidth]{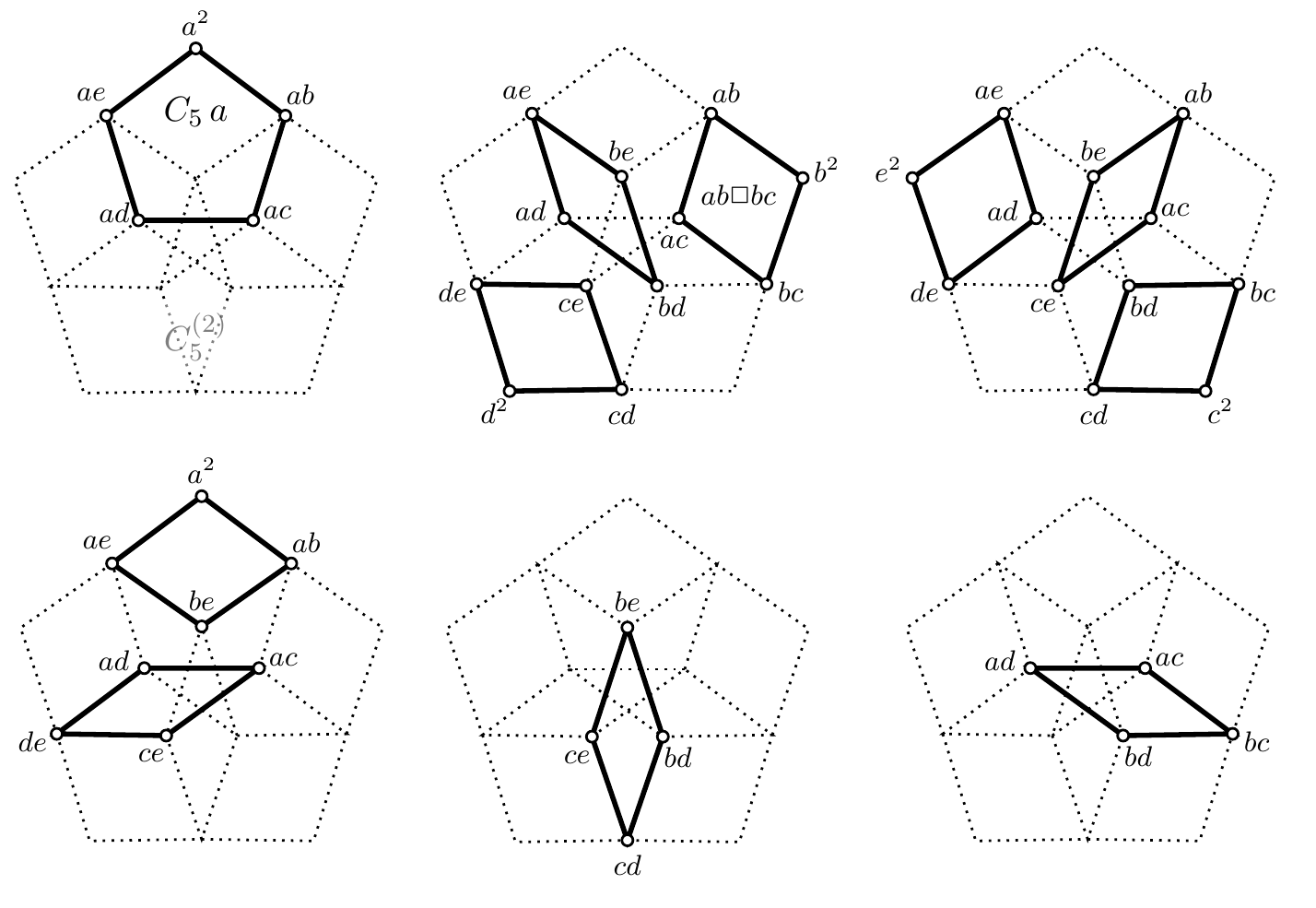}
\caption{The union $\{C_5a\}\cup\mathscr{B}$ is an MCB for
$\mathscr{C}(C_5^{(2)})=\mathscr{C}(C_5\,a)\,\bigoplus\, \mathscr{S}(C_5^{(2)})$.}
\label{Fig:FiveCycleMCB}
\end{figure}

We have seen that  the projections $p_i:G^k\to G$ induce linear maps 
$\mathscr{C}(G^{k})\to \mathscr{C}(G)$.
But there seems to be no obvious way of defining a projection $G^{(k)}\to G$.
 Still,  it is possible to construct a natural linear
map $p^*:\mathscr{C}(G^{(k)})\to \mathscr{C}(G)$. To do this, recall that any edge of
$G^{(k)}$ has form $af\;bf$ where $ab\in E(G)$ and $f\in M_{k-1}(G)$.
We begin by defining $p^*$ on the edge space.
Put $p^*(af\;bf)=ab$ for each edge $af\;bf$ in the basis $E(G^{(k)})$ and extend linearly
to a map $p^*:\mathscr{E}(G^{(k)})\to \mathscr{E}(G)$. Note that
$\sum_{i=1}^{k} p_i^*=p^*\circ \eta^*$. (Confirm it by checking it on
the basis $E(G^k)$ of $\mathscr{E}(G^k)$.) Now, if $X\in \mathscr{C}(G^{(k)})$, then
Lemma~\ref{LEMMA:ETA} guarantees $X=\eta^*(Y)$ for some $Y$ in the cycle space of $G^k$.
Then $p^*(X)=p^*(\eta^*(Y))=\sum_{i=1}^{k} p_i^*(Y)\in \mathscr{C}(G)$.

We now have a linear map $p^*:\mathscr{C}(G^{(k)})\to \mathscr{C}(G)$
for which $p^*(af\;bf)=ab$.

\section{Decomposing the cycle space of a reduced power}
\label{SECTION:DECOMPOSING}

This section explains how to decompose the cycle space of a reduced power into the direct
sum of particularly simple subspaces.

To begin, notice that if $f$ is a fixed monomial in $M_{k-1}(G)$, then
there is an embedding $G\to G^{(k)}$ defined as $x\mapsto xf$. Let us call the image
of this map $Gf$. Notice that $Gf$ is an induced subgraph of $ G^{(k)}$ and is isomorphic
to $G$. 

\begin{proposition}
For any fixed $f\in M_{k-1}(G)$, we have
$\mathscr{C}(G^{(k)})=\mathscr{C}(Gf)\bigoplus \ker(p^*)$.
\label{PROP:SUM}
\end{proposition}
\begin{proof}
Consider the map $p^*:\mathscr{C}(G^{(k)})\to \mathscr{C}(G)$.
Its restriction $\mathscr{C}(Gf)\to \mathscr{C}(G)$ is a vector space isomorphism.
The proof now follows from elementary linear algebra.
\end{proof}

Next we define a special type of cycle in a reduced power.
Given distinct edges $ab$ and $cd$ of $G$ and any $f\in M_{k-2}(G)$,
we have a square  in $G^{(k)}$ with vertices $acf, bcf, bdf, adf$. Let us call such a square
a {\it Cartesian square}, and denote it as $(ab\Box cd)f$. See Figure~\ref{FIG:SQUARE}.
\begin{figure}[h!]
\centering
\includegraphics{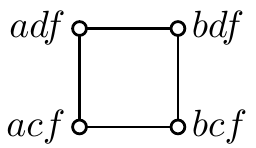}
\caption{A Cartesian square $(ab\Box cd)f$ in $G^{(k)}$ with $k \geq 2$.}
\label{FIG:SQUARE}
\end{figure}

We regard this as
a cycle in the cycle space; it is the subgraph of $G^{(k)}$ that is precisely the sum of edges
$acf\,bcf + bcf\,bdf + bdf\,adf + adf\,acf.$
(Observe that this sum is zero if and only if $ab=cd$.) 
We remark that although a subgraph $Gf$ may have squares, they are not {\it Cartesian squares}
because they do not have the form specified above.
Define the {\it square space} $\mathscr{S}(G^{(k)})$ to be the subspace of $\mathscr{C}(G^{(k)})$
that is spanned by the Cartesian squares. 

If $S$ is a Cartesian square, then $p^*(S)=0$, so $\mathscr{S}(G^{(k)})\subseteq \ker(p^*)$.
In the remainder of the paper we will show that in fact $\mathscr{S}(G^{(k)})= \ker(p^*)$,
so that Proposition~\ref{PROP:SUM} gives
$\mathscr{C}(G^{(k)})=\mathscr{C}(Gf)\bigoplus \mathscr{S}(G^{(k)})$.
Simultaneously we will craft a simple MCB for $G^{(k)}$ by concatenating
MCBs of  $\mathscr{C}(Gf)$ and $\mathscr{S}(G^{(k)})$.
See~Figure~\ref{Fig:FiveCycleMCB} for an example.

\section{Cycle bases for reduced powers}
\label{SECTION:HIGHER}
This section describes a simple cycle basis for the reduced $k$th power of a graph $G$.
If $G$ has no triangles, this cycle basis will be an MCB. (We do not consider MCBs 
in the cases that $G$ has triangles because the applications we have in mind do not involve
such situations. 
Constructing MCBs when $G$ has triangles would be an interesting research problem.)

Let $G$ be a connected graph with $v$ vertices and $e$ edges. Recall that by
Equations~(\ref{EQN:VERTICES}) and~(\ref{EQN:EDGES}), the graph $G^{(k)}$ has
$\binom{k+v-1}{k}$ vertices, identified with the monomials $M_k(G)$, and $e\binom{k+v-2}{k-1}$ edges.
Thus any cycle basis has dimension
\begin{equation}
\beta(G^{(k)})=e\binom{k+v-2}{k-1}-\binom{k+v-1}{k}+1.
\label{EQN:DIMENSION_Gk}
\end{equation}

We first examine the square space. Any pair of distinct edges $ab$ and $cd$ of $G$ corresponds
to a Cartesian square $(ab\Box cd)f$, where $f\in M_{k-2}(G)$, so there are
$\binom{e}{2}\binom{k+v-3}{k-2}$ such squares.
But this set of squares may not be independent. Our first task will be to construct a linearly
independent set of Cartesian squares.

To begin, put $V(G)=\{a_1,a_2,\ldots, a_v\}$.
Let $T$ be a rooted spanning tree of $G$ with root $a_1$, and arrange the indexing
so its order respects a breadth-first traversal of $T$, that is,
for each $i$ the vertex $a_i$ is not closer to the root than any $a_j$ for which
$j<i$ (see Figure~\ref{FIG:TREE}).

\begin{figure}[h]
\centering
\includegraphics{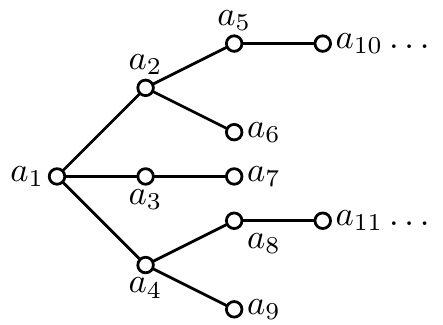}
\caption{A rooted spanning tree $T$ of $G$ with $V(G)=\{a_1,a_2,\ldots, a_v\}$, root $a_1$, and  indexing that 
respects a breadth-first traversal of $T$.}
\label{FIG:TREE}
\end{figure}

With this labeling, any edge of $T$ is uniquely determined by its endpoint $a_j$ that
is furthest from the root. For each $2\leq i\leq v$, let $e_j$ be the edge of $T$ 
that has endpoints $a_i$ and $a_j$, with $a_j$ further from the root than $a_i$.
Let $M_{k-2}(a_1,a_2,\ldots a_j)$ denote the monic monomials of degree $k-2$
in indeterminates $a_1,a_2,\ldots,a_j$, with $1\leq j\leq v$. Define the following sets of
Cartesian squares in  $G^{(k)}$.
\begin{eqnarray*}
 \Upsilon &=& \left\{ (e_i\Box e_j)f \mid 2\leq i<j\leq v, f\in M_{k-2}(a_1,a_2,\ldots, a_j)\right\}, \\
\Omega &=& \left\{ (a_\ell a_m\Box e_j)f  \mid a_\ell a_m\in E(G)\!-\!E(T), 2\leq j\leq v,  f\in M_{k-2}(a_1,a_2,\ldots, a_j)\right\}.
\end{eqnarray*}

Shortly we will show that $\Upsilon\cup\Omega$ is linearly independent. But first a few
quick informal words about why we would expect this to be the case.
Suppose $k\geq 3$ and take three distinct edges
$a_ia_j$, $a_\ell a_m$ and $a_pa_q$ in $G$, and let $f\in M_{k-3}(G)$.
Figure~\ref{FIG:CUBE} indicates that these edges result in a cube in the $k$th
reduced power.
Each of the six square faces of this cube is in the square space.
But the faces are dependent because any one of them is a sum of the others.
Call a square face such as $(a_ia_j\Box a_\ell a_m)a_qf\,$ a  ``top square''
of a cube if the monomial $a_qf$ involves an indeterminate $a_t$ with $t>\max\{i,j,\ell,m\}$. 
Sets $\Upsilon$ and $\Omega$ are constructed so as to contain no top squares.

\begin{figure}[h]
\centering
\includegraphics{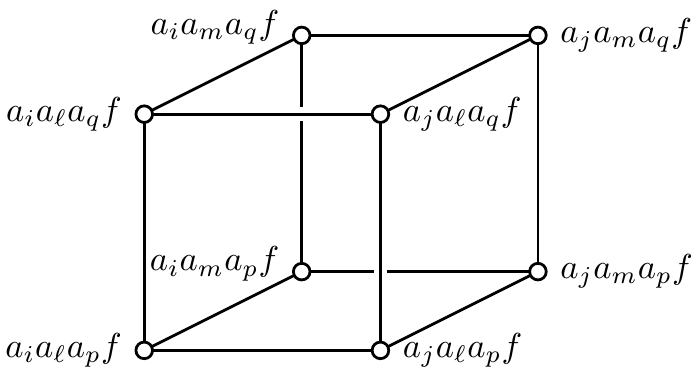}
\caption{A Cartesian cube  $(a_ia_j\Box a_\ell a_m\Box a_pa_q)f$ in the reduced power $G^{(k)}$.
}
\label{FIG:CUBE}
\end{figure}

(A configuration of the type illustrated in Figure~\ref{FIG:CUBE} may not always
be a cube in the combinatorial sense.
The reader is cautioned that if $a_ia_j$, $a_\ell a_m$ and $a_pa_q$ are the edges
of a triangle in $G$, then two of the diagonally opposite vertices of the ``cube'' are the same, as in
$K_3^{(3)}$, shown in Figure~\ref{Fig:Three}. Here there is only one cube, which takes the form
of a central vertex connected to the six vertices of a hexagon.
 This will cause no difficulties in what follows, even
if we entertain the possibility that $G$ does indeed have triangles.)

There is another kind of dependency that is ruled out in the definition of 
$\Upsilon$ and $\Omega$, and we now sketch it. First, imagine $G^2$.
Consider two cycles $A$ and $B$ in $G$ each having exactly one edge not in $T$,
say $a_ia_j$ and $a_\ell a_m$, respectively. Envision
$A\Box B$ is as a torus in $G^2$ with square faces, each edge
shared by two faces. In adding up all the faces, the edges cancel in pairs, giving $0$,
so the squares are dependent. Removing the face $a_ia_j\Box a_\ell a_m$ removes
the dependency. Such squares $a_ia_j\Box a_\ell a_m$ show up in  $G^{(2)}f\subseteq G^{(k)}$ as squares $(a_ia_j\Box a_\ell a_m)f$ with $a_ia_j, a_pa_q\in E(G)-E(T)$.
Sets $\Upsilon$ and $\Omega$ contain no such squares.

\begin{proposition}
The set $\mathscr{B}=\Upsilon \cup \Omega$ is linearly independent.
\label{PROP:IND}
\end{proposition}
\begin{proof}
We first show that $\Upsilon$ is linearly independent. Let $X=\sum (e_i\Box e_j)f_n$ be a sum of elements
of $\Upsilon$. Form the forest $F\subseteq T$ consisting of all edges $e_i$ and $e_j$
that appear as  edges of a squares in this sum, and let $ab$ be an edge of $F$ for which $b$ is a leaf.
Then any term $(a_\ell a_m\Box ab)f_n$ of the sum is the unique square in the sum containing the edge
$a_\ell bf_n\;a_mbf_n$. Because no term can cancel this edge, we get $X\ne 0$, so $\Upsilon$ is linearly
independent.

To see that $\Omega$ is linearly independent, consider a sum
$X=\sum (a_\ell a_m\Box e_j)f_n$ of squares in $\Omega$. Again form a forest 
$F\subseteq T$ of the edges $e_j$ and let $ab$ be as before.
Then any term $(a_\ell a_m \Box ab)f_n$ is the unique square in the sum containing the edge
$a_\ell bf_n\;a_mbf_n$. Then $X\ne 0$ because no other term in the sum can cancel this edge;
hence $\Omega$ is linearly independent.

Now we argue that the spans of $\Upsilon$ and $\Omega$ have
zero intersection. By the previous paragraph, any nonzero linear combination of squares in
$\Omega$ has edges of form  $(a_\ell a_m \Box ab)f_n$, with $a_\ell a_m\in E(G)-E(T)$. But no linear combination
of squares in $\Upsilon$ has such edges. Hence the spans have zero intersection,
so $\mathscr{B}$ is linearly independent.
\end{proof}

Our next task is to show that $\mathscr{B}$ is actually a basis for the square space.
In fact, we will show more: it is also a basis for $\ker(p^*)$, and $\mathscr{S}(G^{(k)})=\ker(p^*)$.
Our dimension counts will involve finding $|\Upsilon|$ and $|\Omega|$, and for this we
use the following formulas. The first is standard; both are easily verified with induction.

\begin{eqnarray}
\mbox{\large $\binom{r}{r}\,+\,\binom{r+1}{r}\;+\;\binom{r+2}{r}\;+\;\cdots\;+\;\,\binom{r+n}{r}$} &=&
\mbox{\large $\,\;\binom{r+n+1}{r+1}$}
\label{EQN:PASCAL1}\\
& & \nonumber \\
\mbox{\large $0 \binom{r}{r}+1\binom{r+1}{r}+2 \binom{r+2}{r}+\;\cdots\;+n \binom{r+n}{r}$} &=& 
\mbox{\large $n \binom{r+n+1}{r+1}-\binom{r+n+1}{r+2}\;\;\;\;\;$}
\label{EQN:PASCAL2}
\end{eqnarray}

\medskip

Take an edge $e_j$ of $T$ with $3\leq j$. From its definition,  $\Upsilon$ has
$(j-2)\binom{k+j-3}{k-2}$ squares of form $(e_i\Box e_j)f$. We reckon as follows, using
Equations~(\ref{EQN:PASCAL1}) and~(\ref{EQN:PASCAL2}) as appropriate.

\begin{eqnarray}
|\Upsilon| & = & \sum_{j=3}^{v}(j-2)\binom{k+j-3}{k-2}
\nonumber\\
& = & \sum_{j=1}^{v}(j-2)\binom{k+j-3}{k-2} +1
\nonumber\\
& = & \sum_{j=1}^{v}(j-1)\binom{k+j-3}{k-2}-\sum_{j=1}^{v}\binom{k+j-3}{k-2} +1
\nonumber\\
& = & (v-1)\binom{k+v-2}{k-1}-\binom{k+v-2}{k}-\binom{k+v-2}{k-1}+1
\nonumber\\
& = & (v-1)\binom{k+v-2}{k-1}-\binom{k+v-1}{k}+1.
\label{EQN:UPSILON}
\end{eqnarray}
Now, given and edge  $e_j$ of $T$ with $2\leq j$,  the set $\Omega$ has
$\beta(G)\binom{k+j-3}{k-2}$ squares of form $(a_\ell a_m\Box e_j)f$.
Consequently
\begin{eqnarray}
|\Omega| & = & \beta(G)\sum_{j=2}^{v}\binom{k+j-3}{k-2}
\nonumber\\
& = & \beta(G)\left(\sum_{j=1}^{v}\binom{k+j-3}{k-2}-1\right)
\nonumber\\
& = & \beta(G)\binom{k+v-2}{k-1}-\beta(G).
\label{EQN:OMEGA}
\end{eqnarray}

\begin{proposition}
The set $\mathscr{B}=\Upsilon\cup\Omega$ is a basis for the square space of the
reduced $k$th power of $G$.
Moreover,
the square space equals $\ker(p^*)$.
\label{PROP:KER}
\end{proposition}
\begin{proof}
By Proposition~\ref{PROP:IND}, the set $\mathscr{B}$ is linearly independent;
and it is a subset of the square space by construction. We saw
earlier that the square space is a subspace $\ker(p^*)$.
To finish the proof we show that $\ker(p^*)$ has dimension $|\mathscr{B}|$.
By the rank theorem applied to the surjective map $p^*:\mathscr{C}(G^{(k)})\to \mathscr{C}(G)$
we have $\dim \ker(\varphi^*)= \beta(G^{(2)})-\beta(G)$.  This with
Equations~(\ref{EQN:DIMENSION_Gk}), (\ref{EQN:UPSILON}) and~(\ref{EQN:OMEGA}),
as well as the fact that $(v-1)+\beta(G)=e$, gives
\begin{eqnarray*}
|\mathscr{B}|& = &|\Upsilon|+|\Omega|\\
& = & (v-1)\binom{k+v-2}{k-1}-\binom{k+v-1}{k}+1+\beta(G)\binom{k+v-2}{k-1}-\beta(G)\\
& = &  e\binom{k+v-2}{k-1}-\binom{k+v-1}{k}+1-\beta(G)\\
& = &  \beta(G^{(2)})-\beta(G)\\
& = & \dim \ker(p^*) .
\end{eqnarray*}
Therefore $\mathscr{B}$ is a basis for both $\mathscr{S}(G^{(k)})$ and $\ker(p^*)$.
\end{proof}

If $k=2$, then $\mathscr{B}=\{ab\Box cd \mid ab, cd\in E(G)\}-\{ab\Box cd \mid ab, cd\in E(G)-E(T)\}$,
so $|\mathscr{B}|=\binom{e}{2}-\binom{\beta(G)}{2}$. It is interesting to note that if $\beta(G)\leq 1$, then 
$\binom{\beta(G)}{2}=0$ and $\mathscr{B}$ consists of all squares in the square space; in all other cases it has fewer squares.

We now can establish the main result of this section, namely a construction of an 
MCB for the reduced $k$th power. Take an $f\in M_{k-1}(G)$. Propositions~\ref{PROP:SUM}
and~\ref{PROP:KER} say
\begin{equation}
\mathscr{C}(G^{(k)})=\mathscr{C}(Gf)\,\bigoplus\, \mathscr{S}(G^{(k)}).
\label{EQN:DECOMP}
\end{equation}
To any cycle $C=c_1c_2\ldots c_n$ in $G$,
 there  corresponds  cycle $Cf=c_1f\, c_2f\ldots \,c_nf$ in 
$G^{(k)}$.

\begin{figure}[t!]
\centering
\includegraphics[width=\textwidth]{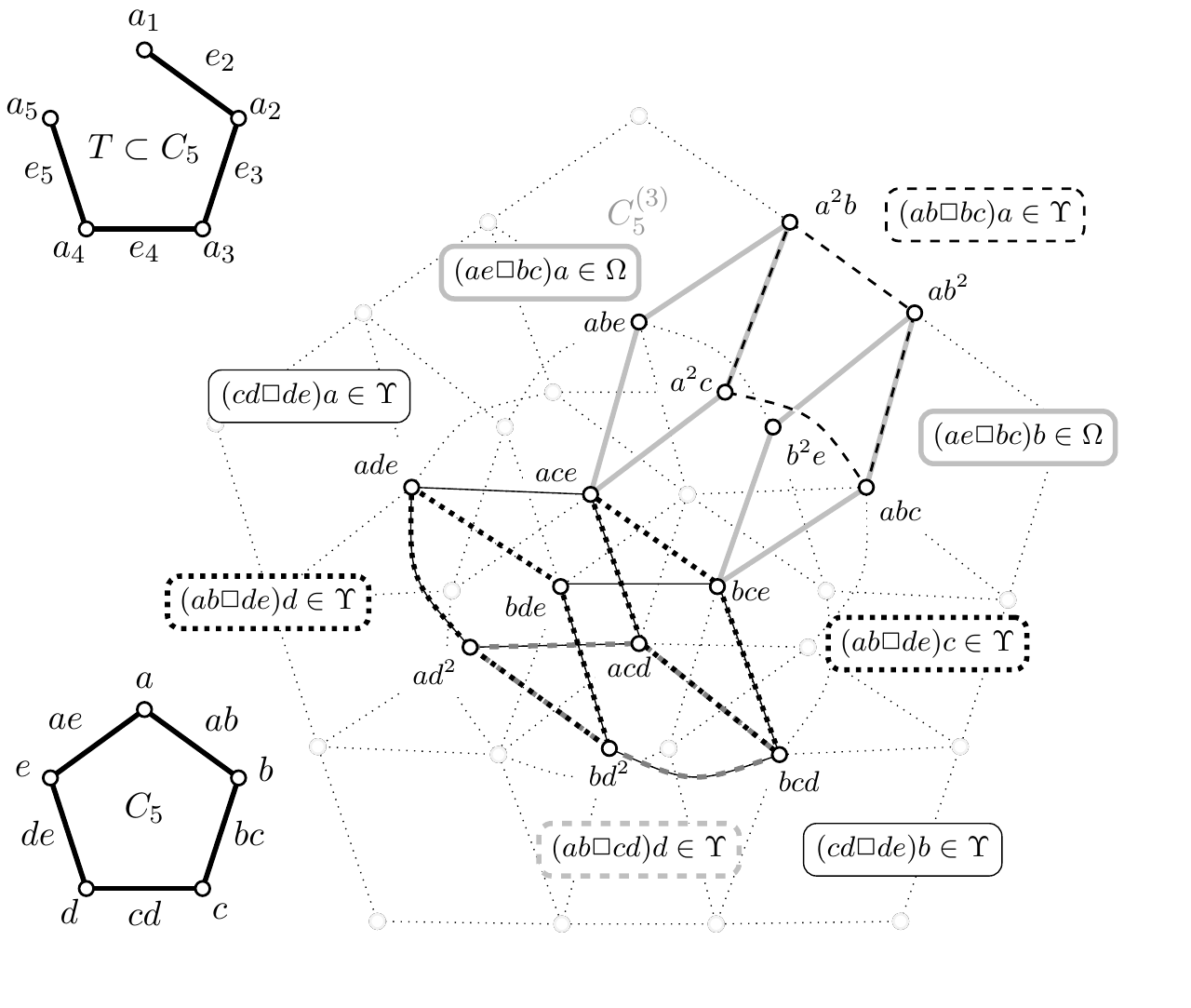}
\caption{With $T$ as indicated, the sets of squares $\Upsilon$ and $\Omega$ form  a basis
$\mathscr{B}=\Upsilon \cup \Omega$ of the square space of $C_5^{(3)}$.
Here
$\Upsilon=$ $\{ (ab\Box bc)f\mid f\in \{ a,b,c \} \} \cup\{ (ab\Box cd)f, (bc\Box cd)f\mid f\in \{ a,b,c,d \} \}$ $\cup \{ (ab\Box de)f, (bc\Box de)f, (cd\Box de) f \mid f\in \{ a,b,c,d,e \} \} $. 
Also $\Omega =$ $\{ (ae\Box ab)f \mid f\in \{ a,b \} \} \cup \{ (ae\Box bc)f \mid f\in \{ a,b,c \} \} \cup \{ (ae\Box cd)f \mid f\in \{ a,b,c,d \} \} \cup \{ (ae\Box de)f \mid f\in \{ a,b,c,d,e \} \}$.
Note $|\Upsilon|=24$ and $|\Omega|=14$.
The square  $(ab \Box cd) e \notin \mathscr{B}$  is the ``top square'' of the Cartesian cube $ab \Box cd \Box de$.}

\label{Fig:FiveCycle(3)Squares}
\end{figure}

\begin{theorem}
Take a cycle basis $\mathscr{C}=\{C_1,C_2,\ldots,C_{\beta(G)}\}$ for $G$, and let
$\mathscr{B}$ be the basis for  $\mathscr{S}(G^{(k)})$  constructed above. Fix
$f\in M_{k-1}(G)$ and put
$\mathscr{C}f=\{C_1f, C_2f,\ldots, C_{\beta(G)}f\}$.
Then $\mathscr{C}f\cup \mathscr{B}$ is a cycle basis for $G^{(k)}$. If $\mathscr{C}$ is an MCB for
$G$, and $G$ has no triangles, then this basis is an MCB for $G^{(k)}$. 
\label{THEOREM:1}
\end{theorem}

\begin{proof}
That this is a cycle basis follows immediately from Equation~(\ref{EQN:DECOMP}).

Now suppose $\mathscr{C}$ is an MCB for
$G$, and that $G$ has no triangles. It is immediate that $G^{(k)}$ has no triangles either.
The proof is finished by applying Proposition~\ref{PROP:TEST}.
Take any
$C\in\mathscr{C}(G^{(k)})$, and write it as
\[C=\sum_{i\in I} G_i+\sum_{j\in J} B_j\,,\]
where the $G_i$ are from $\mathscr{C}f$ and 
the $B_j$ are from $\mathscr{B}$.
According to Proposition~\ref{PROP:TEST}, it suffices to show that $C$
has at least as many edges as any term in this sum.
Certainly $C$ is not shorter than
any square $B_j$ (by the triangle-free assumption).
To see that it is not shorter than any $G_i$ in the sum,
apply $p^*$ to the above equation to get
\[p^*(C)=\sum_{i\in I} p^*(G_i)\,.\]
Because $p^*:\mathscr{C}(Gf)\to \mathscr{C}(G)$ is an isomorphism,
the terms $p^*(G_i)$ are part of an MCB for $G$, and thus
$|p^*(C)|\geq |p^*(G_i)|=|G_i|$ for each $i$, by Proposition~\ref{PROP:TEST}.
Also $|C|\geq |p^*(C)|$ (as some edges may cancel in the projection)
so  $|C|\geq |G_i|$.
\end{proof}

Although Theorem~\ref{THEOREM:1} gives a simple MCB for reduced powers of a graph that has
no triangles, the constructed basis is definitely {\it not} minimum if triangles are present. Several different phenomena
account for this. Consider the case $k=2$. First, if $G$ has triangles, then for each vertex $x$ of $G$, the second reduced power contains
a copy $Gx$ of $G$. These copies are pairwise edge-disjoint; an MCB would have to capitalize on triangles in each of these copies
at the expense of squares in the square space. Moreover, as Figure~\ref{Fig:Three} demonstrates, some of the squares
in the square space will actually be sums of two triangles. The figure also shows that for a triangle $\Delta=abc$ in $G$,
we do not get just the three triangles $\Delta a$, $\Delta b$ and $\Delta c$, but also a fourth triangle $ab\,bc\,ca$ not
belonging to any $Gx$. We do not delve into this problem here.

\section{Discussion}
\label{Sec:Discussion}

We have defined what appears to be a new construction,  the $k$th reduced power of a graph, $G^{(k)}$, and have presented a theorem for construction of minimal cycle bases of~$G^{(k)}$.  

When $G$ is the transition graph for a Markov chain, $G^{(k)}$
is  the transition graph for the configuration space of $k$ identical and indistinguishable $v$-state automata with transition graph $G$.  Symmetry of model composition allows for interactions among  stochastic automata, so long as the transition rates $q_{ij}$ for $i,j \in \{ 1,2, \cdots , v\}$, $i\neq j$ are constant or functions of the number of automata $n_\ell(t)$ in each state, $0 \leq n_\ell(t) \leq k$, $1 \leq \ell \leq v$.  $G^{(k)}$ does not pertain if transition rates depend on the state of any particular automaton, $X_n(t) \in \{ 1, 2, \cdots , v\}$, $n \in \{ 1,2, \cdots , k\}$, as this violates indistinguishability. 

For concreteness, consider a stochastic automata network composed of three identical automata, each with transition graph $C_5$ and generator matrix, 
\bne
Q = \left( \begin{array}{ccccc} 
\diamond & q_{ab}[\cdot] & 0 & 0& q_{ae}\\
q_{ba} & \diamond  & q_{bc} &0 & 0\\
 0 &  q_{cb}& \diamond  & q_{cd}& 0  \\
0 & 0  &  q_{dc} & \diamond& q_{de}\\
q_{ea}  & 0  &0  & q_{ed} &  \diamond \\
\end{array} \right) 
\label{EQ:Q5}
\ene
where $\diamond$'s indicate the values required for zero row sum, $q_{ii} = - \sum_{j\neq i} q_{ij} < 0 $, and  $q_{ab}[\cdot]$ indicates a functional transition rate that depends on the global state of the three automata.  Assume constant transition rates $q_{bc} = q_{cd} = q_{de}  = q_{ea}  = \mu>0$ and $q_{ba}=q_{cb}=q_{dc}=q_{ed}=q_{ae}=\nu>0$.  Further assume that 
the automata may influence one another through the state-dependent  transition rate,
\bne
q_{ab}[\cdot]  = \lambda+ \alpha \, (n_a[\cdot] - 1) + \beta \, n_b[\cdot]  + \gamma \, n_c[\cdot]  + \delta \, n_d[\cdot]  + \epsilon \, n_e[\cdot] ,
\label{EQ:QAB}
\ene
where $\alpha,\beta,\gamma,\delta, \epsilon \geq 0$ and $[\cdot]$ denotes the global state $a_1^{p_1}a_2^{p_2}\cdots a_v^{p_v}$ that is the functional transition rate's argument.  The transition rate $q_{ab}: M_k(a_1,a_2,\cdots, a_v) \rightarrow \mathbb{R}$ is a function of the global state via $n_\ell : M_k(a_1,a_2,\cdots, a_v) \rightarrow \mathbb{N}$ defined by $n_\ell[a_1^{p_1}a_2^{p_2}\cdots a_v^{p_v}]=p_\ell$.  The three automata are uncoupled when $\alpha,\beta,\gamma,\delta,\epsilon = 0$ because this eliminates the dependence of $q_{ab}[\cdot]$ on the global state.  

(In this model specification, coupling an isolated component automaton to itself is equivalent to absence of coupling.  Because $q_{ab}[\cdot]$ is the rate of an $a \rightarrow b$ transition, $q_{ab}[\cdot]$ is only relevant when the isolated automaton is in state $a$.  This functional transition rate has the property that $q_{ab}[a]=\lambda$ when $\alpha,\beta,\gamma,\delta, \epsilon > 0$ because $n_x[y] = 1$ for $x=y$ and 0~otherwise.) 

The transition matrix for the master Markov chain $Q^{(3)}$ is defined by the model specification in the previous paragraph.  For example, the transition rate from global state $ad^2$ to global state $abd$ is $q^{(3)}[ad^2,abd] =2\mu$ because $n_d[ad^2] =2$ and $q_{db}=\mu$ is not a function of the global state.  Other examples are $q^{(3)}[c^3,c^2d]=n_c[c^3] q_{cd}=3\nu$, $q[a^2c,a^2d]=n_c[a^2c]q_{cd} = \nu$,   
\bnea
q^{(3)}[abe,b^2e] &= & n_a[abe] q_{ab} [abe] \nonumber \\
&=&  \lambda+ \alpha \, (n_a[abe]-1) + \beta \, n_b[abe]  + \gamma \, n_c[abe] + \delta \, n_d[abe]  + \epsilon \, n_e[abe]  \nonumber \\
&=&   \lambda+ \beta  + \epsilon \nonumber \\
q^{(3)}[a^3,a^2b] &= &  n_a[a^3] q_{ab} [a^3]  \nonumber \\
&=& 3\left(  \lambda+ \alpha \, (n_a[a^3]-1) + \beta \, n_b[a^3]  + \gamma \, n_c[a^3] + \delta \, n_d[a^3]  + \epsilon \, n_e[a^3] \right) \nonumber \\
&=& 3\left(  \lambda+ 2\alpha  \right)    \nonumber \\
q^{(3)}[a^2c,abc] &= & n_a[a^2c] q_{ab}[a^2c] \nonumber \\
&=& 2\left(  \lambda+ \alpha \, (n_a[a^2c] -1)+ \beta \, n_b[a^2c]  + \gamma \, n_c[a^2c]  + \delta \, n_d[a^2c]  + \epsilon \, n_e[a^2c]  \right) \nonumber \\
&=& 2 \left(  \lambda+  \alpha + \gamma \right)   . \nonumber
\enea
This process of unpacking the model specification yields a  master  Markov chain with $\eta = \binom{k+v-1}{k} =  \binom{3+5-1}{3} = 35$ states.  The master Markov chain has 210 transition rates $q_{ij}>0$ corresponding (in pairs) to the $5\binom{3+5-2}{3-1}=105$ edges of the master transition graph $C_5^{(3)}$.

\begin{figure}[t]
\centering
\includegraphics[width=\textwidth]{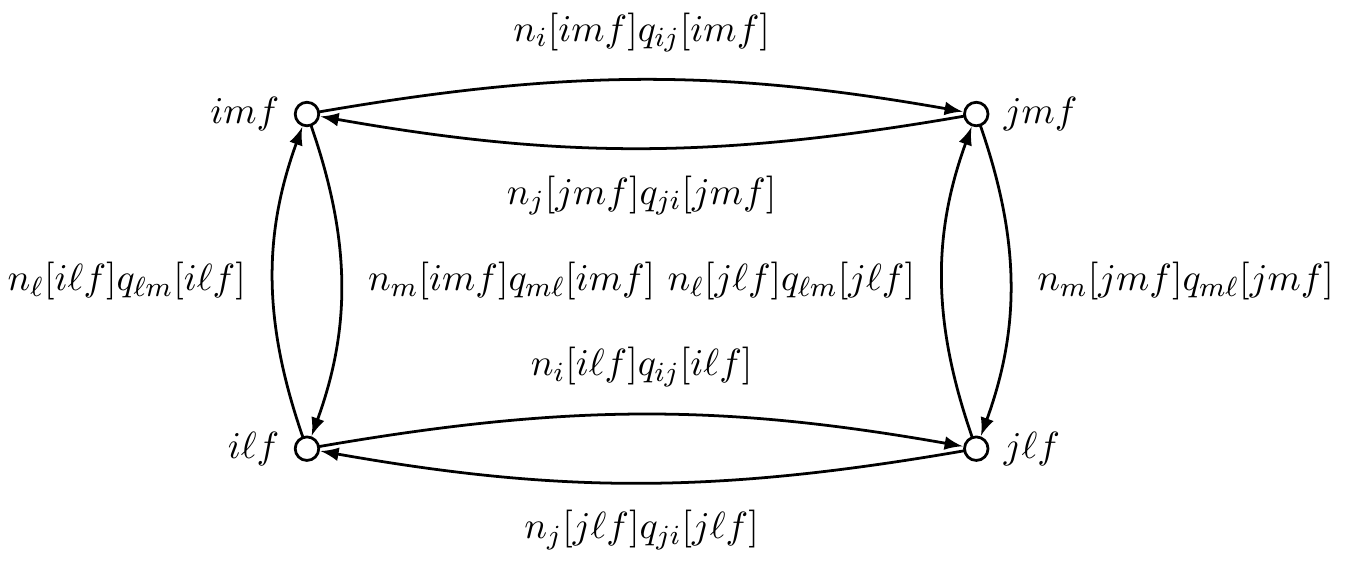}
\caption{Many cycles of the directed, weighted transition graph for a master Markov chain for $k$ coupled $v$-state automata correspond to Cartesian squares $(ij\Box \ell m )f$ of the minimal cycle basis for  the undirected, unweighted transition graph $G^{(k)}$, where $i,j,\ell,m \in \{ a_1,a_2,\cdots, a_v \}$ and $f \in M_{k-2}(a_1,a_2,\cdots, a_v)$.}
\label{FIG:SQUAREKOMOLGOROV}
\end{figure}

The construction of 
minimal cycle bases of $G^{(k)}$ provided by Theorem~\ref{THEOREM:1} is especially relevant to stochastic automata networks that arise in physical chemistry and biophysics \cite{Hill89}.  For many  applications in these domains, the principle of microscopic reversibility requires that the stationary distribution of {\it uncoupled} automata satisfying global balance,   
$\bar{\bpi} Q = \bzero$ subject to $\sum_i \bar{\pi}_i = 1$, also satisfies a stronger condition known as detailed balance, 
\[
\bar{\pi}_i \sum_{i \neq j} q_{ij} = \sum_{j \neq i} q_{ji} \bar{\pi}_j    .
\]
In other words, nonequilibrium steady states are forbidden.  Markov chains have this property when the transition rates satisfy the Kolmogorov criterion, namely, equality of the product of rate constants in both directions around any cycle in the transition matrix $Q$   \cite{Kelly11}. For an isolated automaton with transition graph $C_5$ and transition matrix (\ref{EQ:Q5}), the Komologorov criterion is 
\bne
q_{ab} [a]\,  q_{bc} \, q_{cd} \,  q_{de} \, q_{ea} = q_{ae} \,  q_{ed} \, q_{dc} \,  q_{cb} \,  q_{ba}  .
\label{KomolgorovB}
\ene
Substituting the transition rates of the model specification, both those that are constant as well as $q_{ab}[a]=\lambda$ (\ref{EQ:QAB}), yields 
the following condition on model parameters,
\bne
\lambda  \mu^4   = \nu^5 ,
\label{KomolgorovC}
\ene
that ensures the stationary distribution of an isolated automaton will satisfy detailed balance. 

By constructing the minimal cycle basis of  $C_5^{(3)}$, we may verify that the master Markov chain for three {\em uncoupled} automata, each with transition graph $C_5$, also exhibits microscopic reversibility under the same parameter constraints.   

To see this, recall that the minimal cycle basis of $C_5^{(3)}$ has 39 linearly independent cycles.  Microscopic reversibility for the master Markov chain for three uncoupled  automata 
requires that, given (\ref{KomolgorovC}) and $\alpha,\beta,\gamma,\delta,\epsilon=0$,  39  Komolgorov criteria are satisfied, each corresponding to  a $C_i$ in the MCB for $C_5^{(3)}$. 

One cycle in the MCB for $C_5^{(3)}$ takes the form $C_5 f$ for fixed  $f\in M_2(a,b,c,d,e)$. 
The Kolmogorov criterion for this cycle is 
\bnea
&& n_a[af] q_{ab}[af] \cdot n_b[bf] q_{bc}[bf] \cdot n_c[cf]q_{cd}[cf] \cdot n_d[df]q_{de}[df] \cdot n_e[ef]q_{ea}[ef] \nonumber \\
&& =  n_a[af]q_{ae}[af] \cdot n_e[ef]q_{ed}[ef] \cdot n_d[df]q_{dc}[df] \cdot n_c[cf]q_{cb}[cf] \cdot n_b[bf]q_{ba}[bf]  \nonumber  ,
\enea
where, for typographical efficiency, here and below, we drop the superscripted $(3)$ on the transition rates $q^{(3)}[\cdot,\cdot]$ of $Q^{(3)}$.  Canceling identical terms of the form $n_x[xf]$ gives 
\[
q_{ab}[af] \cdot q_{bc}[bf] \cdot q_{cd}[cf] \cdot q_{de}[df] \cdot q_{ea}[ef] 
=  q_{ae}[af] \cdot q_{ed}[ef] \cdot q_{dc}[df] \cdot q_{cb}[cf] \cdot q_{ba}[bf] . \]
When this expression is evaluated, the result is another instance of  (\ref{KomolgorovC}), which is satisfied by assumption. 

The remaining 38 $C_i$ in the MCB for $C_5^{(3)}$ are Cartesian squares (see Figure~\ref{FIG:SQUAREKOMOLGOROV}) that yield Kolmogorov criteria of the form,
\bnea 
&& n_i[imf] q_{ij} [imf]  \cdot n_m[jmf] q_{m\ell} [jmf] \cdot n_j[j\ell f] q_{ji} [j\ell f] \cdot n_\ell [i\ell f] q_{\ell m} [i\ell f]  \nonumber \\
&& = n_m[imf] q_{m\ell} [imf]  \cdot n_i[i\ell f] q_{ij} [i\ell f]  \cdot n_\ell[j\ell f] q_{\ell m} [j\ell f]  \cdot n_j[jmf] q_{ji} [jmf] , \nonumber 
\enea 
where $f\in M_1(a,b,c,d,e)$. 
For $x\neq y$,  $n_x [xyf] = n_x [x] + n_x [y] + n_x [f] =1 + n_x [f]$, so~this criterion simplifies to
\bnea 
&& \hspace{-0.3in} (1+n_i[f]) q_{ij} [imf]  \cdot (1+ n_m[f])q_{m\ell} [jmf] \cdot (1+n_j[f]) q_{ji} [j\ell f] \cdot (1+ n_\ell [f]) q_{\ell m} [i\ell f]  \nonumber \\
&& \hspace{-0.3in}  = (1+n_m[f] ) q_{m\ell} [imf]  \cdot (1+n_i[f]) q_{ij} [i\ell f]  \cdot (1+ n_\ell[f]) q_{\ell m} [j\ell f]  \cdot (1+ n_j[f]) q_{ji} [jmf] . \nonumber 
\enea 
Canceling identical terms of the form $(1+n_x[f])$ gives 
\bne
 q_{ij} [imf]  \, q_{m\ell} [jmf] \,  q_{ji} [j\ell f] \, q_{\ell m} [i\ell f]  =  q_{m\ell} [imf] \,  q_{ij} [i\ell f]  \,  q_{\ell m} [j\ell f]  \,q_{ji} [jmf]  
\label{KomolgorovE}
\ene
for $(ij \Box \ell m )f \in \mathscr{B}  = \Upsilon \cup \Omega$ with $f \in M_{1} (a_1,a_2,\ldots,a_v)$.
When the  automata are not coupled, $\alpha,\beta,\gamma,\delta,\epsilon=0$, the transition rates are not functions of the global state, and every factor on the left hand side has an equal partner on the right. 
Consequently, the 38 squares of $\mathscr{B}$  correspond to cycles in $Q^{(3)}$ that satisfy Komolgorov criteria.   

We have shown that every cycle in the MCB for $C_5^{(3)}$, given by $C_5 a \cup \mathscr{B}$, corresponds to a cycle in $Q^{(3)}$ that satisfies a Komolgorov criterion.   For every cycle in $Q^{(3)}$, there is a representative in the cycle space  $\mathscr{C}(C_5^{(3)})$ that is a linear combination (over the field $\mathbb{F}_2$) of elements of the MCB.  It follows that 
every cycle in the master Markov chain satisfies the Komolgorov criterion.  Thus, we conclude that 
the master Markov chain for three {\em uncoupled} automata
exhibits microscopic reversibility provided an isolated automaton has this property.  This property is expected, and yet important for  model verification. 

In many applications, it is important to establish whether or not model composition  (i.e., the process of coupling the automata) results in a master Markov chain with nonequilibrium steady states, in spite of the fact that an isolated component automaton satisfies detailed balance.   Such nonequilibrium steady states may be objects of study or, alternatively, the question may be relevant  because the master Markov chain is not physically meaningful when model composition introduces the possibility of nonequilibrium steady states \cite{Hill89}.

Our construction of 
minimal cycle bases of reduced graph powers provides conditions sufficient to ensure that model composition does not introduce nonequilibrium steady states.  
In general, it is sufficient that  (\ref{KomolgorovE}) hold of every Cartesian square 
$(ij\Box \ell m )f$ of the MCB for  the undirected, unweighted transition graph $G^{(k)}$.   In the example under discussion, many of these Komolgorov criteria do not involve the functional transition rate $q_{ab}[\cdot]$; these conditions are satisfied without placing any constraints on the coupling parameters $\alpha,\beta,\gamma,\delta,\epsilon$. 
The remaining constraints take the form 
\bne
q_{ab} [amf]  \, q_{m\ell} [bmf] \,  q_{ba} [b\ell f] \, q_{\ell m} [a\ell f] =  q_{m\ell} [amf] \,  q_{ab} [a\ell f]  \,  q_{\ell m} [b\ell f]  \,q_{ba} [bmf] 
\label{KomolgorovF}
\ene 
for $\ell m \in \{ bc,cd,de, ae\}$.   The Cartesian squares of concern are elements of the set $\{ (ab\Box \ell m )f \mid \ell m \in  \{ bc,cd,de\} \} \subset \Upsilon$  and $(ae\Box ab )f \in \Omega$.
Note that  $\ell m \neq ab$ and, consequently, $q_{m\ell} [bmf]  = q_{m\ell} [amf]$, $q_{\ell m} [a\ell f] = q_{\ell m} [b\ell f] $ and 
$q_{ba} [b\ell f] =  q_{ba} [bmf]=\nu$.  Thus, (\ref{KomolgorovF}) simplifies to 
\bne
q_{ab} [a\ell f]  = q_{ab} [amf]   \quad\quad \ell m \in \{ bc,cd,de, ae\}.
\label{KomolgorovG}
\ene 
To see how this requirement constrains the coupling parameters $\alpha,\beta,\gamma,\delta,\epsilon$, we expand both sides of (\ref{KomolgorovG}) using (\ref{EQ:QAB}), for example, 
\bnea
q_{ab} [a\ell f]  &= &\lambda + \alpha (n_a[a\ell f]-1) + \beta n_b[a\ell f] + \gamma n_c[a\ell f]  + \delta n_d[a\ell f]   + \epsilon n_e[a\ell f]   \nonumber \\
&= &\lambda + \alpha n_a[\ell f] + \beta n_b[\ell f] + \gamma n_c[\ell f]  + \delta n_d[\ell f]   + \epsilon n_e[\ell f]   \nonumber
\enea
where we used $n_a[a\ell f]=1 + n_a[\ell f]$.
Subtracting both sides of  (\ref{KomolgorovG}) by $\lambda + \alpha n_a[f]  + \beta n_b[f] + \gamma n_c[f]  + \delta n_d[f]   + \epsilon n_e[f] $  and using $n_x[\ell f]=n_x[\ell] + n_x[f]$ we obtain
\[
\alpha n_a[\ell] + \beta n_b[\ell] + \gamma n_c[\ell]    + \delta n_d[\ell]   + \epsilon n_e[\ell]   =
 \alpha n_a[m]  + \beta n_b[m] + \gamma n_c[m]  + \delta n_d[m]   + \epsilon n_e[m] 
\]
for $ \ell m \in \{ bc,cd,de, ae\}$.
These four equations yield four parameter constraints that ensure detailed balance in the master Markov chain for the three coupled stochastic automata, for example, $\ell m = bc$ gives 
\[
\alpha n_a[b] + \beta n_b[b] + \gamma n_c[b]    + \delta n_d[b]   + \epsilon n_e[b]   =
 \alpha n_a[c]  + \beta n_b[c] + \gamma n_c[c]  + \delta n_d[c]   + \epsilon n_e[c] , 
\]
which implies that $\beta = \gamma$. 
Substituting $\ell m = cd$, $de$ and $ae$, we find $\gamma=\delta$, $\delta=\epsilon$ and $\alpha=\epsilon$, respectively.
We conclude that $\alpha=\beta=\gamma=\delta=\epsilon$. 

In our example, the three automata are coupled when one or more of $\alpha,\beta,\gamma,\delta,\epsilon$ is positive. 
The analysis of Cartesian squares in the MCB for $C_5^{(3)}$ shows that coupling the three automata in the manner specified by (\ref{EQ:QAB}) {\em will} introduce nonequilibrium steady states unless the coupling parameters are equal.  This result is intuitive because $\sum_i n_i[\cdot] = k = 3$ and, consequently, equal  coupling parameters  $\alpha=\beta=\gamma=\delta=\epsilon$ correspond to a functional transition rate that, for every global state, evaluates to the constant  $q_{ab}[\cdot]=\lambda+\alpha(k-1)=\lambda+2\alpha$.

The simplicity of this parameter constraint is a consequence of evaluating (\ref{KomolgorovE}) in the context of the example model specification.  In general, the resulting constraints may be more complex and less restrictive.  Any choice of model parameters that  simultaneously satisfies 
\[
 q_{ij} [imf]  \, q_{m\ell} [jmf] \,  q_{ji} [j\ell f] \, q_{\ell m} [i\ell f]  =  q_{m\ell} [imf] \,  q_{ij} [i\ell f]  \,  q_{\ell m} [j\ell f]  \,q_{ji} [jmf]  
\]
for $(ij \Box \ell m )f \in \mathscr{B}  = \Upsilon \cup \Omega$ with $f \in M_{k-2} (a_1,a_2,\ldots,a_v)$
are conditions sufficient  to ensure that the process of model composition (i.e., coupling $k$ identical and indistinguishable $v$-state automata) does not introduce a violation of microscopic reversibility. 

\section*{Acknowledgments}

The work was supported  by National Science Foundation Grant DMS 1121606.  GDS acknowledges a number of stimulating conversations with William \& Mary students enrolled in Spring 2015 
 {\it Mathematical Physiology} and Professor Peter Kemper.

\bibliography{richard_rpg,math,smith_rpg}
\bibliographystyle{unsrt}

\end{document}